\documentclass[a4paper,12pt]{article}
\usepackage[utf8]{inputenc}
\usepackage[T1]{fontenc}
\usepackage{times}
\usepackage[bitstream-charter]{mathdesign}
\usepackage[all]{xy}
\usepackage{enumerate}
\usepackage{graphicx}
\RequirePackage{amsopn} 
\RequirePackage{amsbsy} 
\RequirePackage{amsmath} 
\RequirePackage{relsize} 
\RequirePackage{latexsym} 
\RequirePackage{theorem}%
\RequirePackage{thc}%
\newcounter{mysubsection}[section]%
\newcounter{mysubsubsection}[mysubsection]%
\theorembodyfont{\slshape}%
\newtheorem{corollary}[mysubsection]{Corollary}%
\newtheorem{lemma}[mysubsection]{Lemma}%
\newtheorem{problem}[mysubsection]{Problem}%
\newtheorem{proposition}[mysubsection]{Proposition}%
\newtheorem{theorem}[mysubsection]{Theorem}%
\theorembodyfont{\upshape}%
\newtheorem{example}[mysubsection]{Example}%
\def\qed{{\unskip\nobreak\hfil\penalty50%
  \hskip2em\hbox{}\nobreak\hfil$\square$%
  \parfillskip=0pt\finalhyphendemerits=0\par}}
\newenvironment{proof}%
  {\par\addvspace{\medskipamount}%
    \upshape%
    {\slshape\scshape
    Proof\hskip\labelsep}}%
  {\qed%
    \addvspace{\medskipamount}}%
\newcommand\ideal[1]{\mathfrak{\lowercase{#1}}}
\newcommand\rMod{\textbf{Mod}-}
\newcommand\Ann{\textrm{Ann}}
\newcommand\Clt[3]{\mathrm{Cl}_{#1}^{#2}(#3)} %
\newcommand\Jac{\textrm{Jac}}

\newcommand\Max{\textrm{Max}}
\newcommand\Reg{\textrm{Reg}}
\newcommand\sAm{\sigma_{A\setminus\ideal{m}}}
\newcommand\sAp{\sigma_{A\setminus\ideal{p}}}
\newcommand\sepa{\vspace*{-3mm}\setlength{\itemsep}{-2mm}}

\newcommand\Spec{\textrm{Spec}}

\newcommand\blfootnote[1]{%
  \begingroup
  \renewcommand\thefootnote{}\footnote{#1}%
  \addtocounter{footnote}{-1}%
  \endgroup}

\begin{document}
\title{An extension of S-artinian rings and modules to a hereditary torsion theory setting}
\author{P. Jara\\ {\normalsize Department of Algebra}\\ {\normalsize University of Granada}}
\date{}
\maketitle\blfootnote{ \date{\today}; pjara@ugr.es
\newline
}

\begin{abstract}
For any commutative ring $A$ we introduce a generalization of $S$--artinian rings using a hereditary torsion theory $\sigma$ instead of a multiplicative closed subset $S\subseteq{A}$. It is proved that if $A$ is a totally $\sigma$--artinian ring, then $\sigma$ must be of finite type, and $A$ is totally $\sigma$--noetherian.
\end{abstract}

\section*{Introduction}

In \cite{Hamann/Houston/Johnson:1988}, the authors study the problem of determining the structure of the polynomial ring $D[X]$ over  an integral domain $D$ with field of fractions $K$, looking for the structure of the Euclidean domain $K[X]$. In particular, an ideal $\ideal{a}\subseteq{D[X]}$ is said to be \textbf{almost principal} whenever there exist a polynomial $F\in\ideal{a}$, of positive degree, and an element $0\neq{s}\in{D}$ such that $\ideal{a}s\subseteq{FD[X]}\subseteq\ideal{a}$. The integral domain $D$ is an \textbf{almost principal domain} whenever every ideal $\ideal{a}\subseteq{D[X]}$, which extends properly to $K[X]$, is almost principal.
Noetherian and integrally closed domains are examples of almost principal domains.

Later, in \cite{Anderson/Dumitrescu:2002}, the authors extend this notion to non--necessarily integral domains in defining, for a given multiplicatively closed subset $S\subseteq{A}$ of a ring $A$, an ideal $\ideal{a}\subseteq{A}$ to be \textbf{$S$--finite} if there exist a finitely generated ideal $\ideal{a}'\subseteq\ideal{a}$ and an element $s\in{S}$ such that $\ideal{a}s\subseteq\ideal{a}'$, and define a ring $A$ to be \textbf{$S$--noetherian} whenever every ideal $\ideal{a}\subseteq{A}$ is $S$--finite.
Many authors have worked on $S$--noetherian rings and related notions, and shown relevant results on its structure. See for instance
\cite{Eljeri:2018,
Hamed:2018,
Lim:2015,
Lim/Oh:2014,
Sevim/Tekir/Koc:2020,
Zhongkui:2007}.

In \cite{Sevim/Tekir/Koc:2020}, the author study $S$--artinian rings, dualizing the former notion of $S$--noetherian ring, and give some characterization of $S$--artinian rings in terms of finite cogeneration with respect to $S$. Our aim is to show that this theory is part of a more general theory involving hereditary torsion theories. In particular, we show that if $A$ is totally $\sigma$--artinian, then the hereditary torsion theory $\sigma$ is of finite type, and, in addition, $A$ it is totally $\sigma$--noetherian.

The background will be the hereditary torsion theories on a commutative (and unitary) ring $A$, see \cite{GOLAN:1986, STENSTROM:1975}, and $\rMod{A}$ denotes the category of $A$--modules. Thus, a hereditary torsion theory $\sigma$ in $\rMod{A}$ is given by one of the following objects:
\begin{enumerate}[(1)]\sepa
\item
a \textbf{torsion class} $\mathcal{T}_\sigma$, a class of modules which is closed under submodules, homomorphic images, direct sums and group extensions,
\item
a \textbf{torsionfree class} $\mathcal{F}_\sigma$, a class of modules which is closed under submodules, essential extensions, direct products and group extensions,
\item
a \textbf{Gabriel filter} of ideals $\mathcal{L}(\sigma)$, a non--empty filter of ideals satisfying that every $\ideal{b}\subseteq{A}$, for which there exists an ideal $\ideal{a}\in\mathcal{L}(\sigma)$ such that $(\ideal{b}:a)\in\mathcal{L}(\sigma)$, for every $a\in\ideal{a}$, belongs to $\mathcal{L}(\sigma)$.
\item
a \textbf{left exact kernel functor} $\sigma:\rMod{A}\longrightarrow\rMod{A}$.
\end{enumerate}
The relationships between these notions are the following. If $\sigma$ is the left exact kernel functor, then
\[
\begin{array}{ll}
\mathcal{T}_\sigma=\{M\in\rMod{A}\mid\;\sigma{M}=M\},\\
\mathcal{F}_\sigma=\{M\in\rMod{A}\mid\;\sigma{M}=0\},\\
\mathcal{L}(\sigma)=\{\ideal{a}\subseteq{A}\mid\;A/\ideal{a}\in\mathcal{T}_\sigma\}.\\
\end{array}
\]
If $\mathcal{L}$ is the Gabriel filter of $\sigma$, and $\mathcal{T}$ the torsion class, for any $A$--module $M$ we have:
\[
\begin{array}{lll}
\sigma{M}
&=\{m\in{M}\mid\;(0:m)\in\mathcal{L}\}
&=\sum\{N\subseteq{M}\mid\;N\in\mathcal{T}\}.
\end{array}
\]

\begin{example}\label{ex:20200103}
\begin{enumerate}[(1)]\sepa
\item
Let $\Sigma\subseteq{A}$ be a multiplicatively closed subset, there exists a hereditary torsion theory, $\sigma_\Sigma$, defined by
\[
\mathcal{L}(\sigma_\Sigma)=\{\ideal{a}\subseteq{A}\mid\;\ideal{a}\cap\Sigma\neq\varnothing\}.
\]
Observe that $\sigma_\Sigma$ has a filter basis constituted by principal ideals. Every hereditary torsion theory $\sigma$ such that $\mathcal{L}(\sigma)$ has a filter basis of principal ideals is called a \textbf{principal} hereditary torsion theory.
We can show that there is a correspondence between principal hereditary torsion theories in $\rMod{A}$, and saturated multiplicatively closed subsets in $A$.
\item\label{it:ex:20200103-2}
Let $\mathcal{A}$ be a set of finitely generated ideals of a ring $A$, then
$$
\mathcal{L}=\{\ideal{b}\subseteq{A}\mid\;\textrm{there exists }\ideal{a}_1,\ldots,\ideal{a}_t\in\mathcal{A}\textrm{ such that }\ideal{a}_1\cdots\ideal{a}_t\subseteq\ideal{b}\}
$$
is a Gabriel filter.
\end{enumerate}
\end{example}

\medskip

This paper is organized in sections. In the first one we introduce the main subject: totally $\sigma$--artinian rings and modules, and show examples, their first properties, and the decisive fact: if $A$ is a totally $\sigma$--artinian ring, then $\sigma$ is a finite type hereditary torsion theory. In section two we deal with scalar extensions, which will be useful for studying local properties. In section three we give an extra characterization of totally $\sigma$--artinian rings and modules with the minimal conditions we found out. In the fourth section, we study we study the behaviour of prime ideals in relation with totally $\sigma$--artinian modules. Sections five and six is devoted to establish the necessary background  to show that every totally $\sigma$--artinian rings is also totally $\sigma$--noetherian.

\section{Totally $\sigma$--artinian rings and modules}

For any $\sigma$--torsion finitely generated $A$--module $M$, if $M=m_1A+\cdots+m_tA$, since $(0:m_i)\in\mathcal{L}(\sigma)$, for any $i=1,\ldots,t$, then $\ideal{h}:=\cap_{i=1}^t(0:m_i)\in\mathcal{L}(\sigma)$, and satisfies $M\ideal{h}=0$. In general, this result does not hold for $\sigma$--torsion non--finitely generated $A$--modules. Therefore, we shall define an $A$--module $M$ to be \textbf{totally $\sigma$--torsion} whenever there exists $\ideal{h}\in\mathcal{L}(\sigma)$ such that $M\ideal{h}=0$. The notion of totally torsion appears, for instance, in \cite[page 462]{Jategaonkar:1971}.

For any ideal $\ideal{a}\subseteq{A}$ we have two different notions of finitely generated ideals relative to $\sigma$:
\begin{enumerate}[(1)]\sepa
\item
$\ideal{a}\subseteq{A}$ is \textbf{$\sigma$--finitely generated} whenever there exists a finitely generated ideal $\ideal{a}'\subseteq\ideal{a}$ such that $\ideal{a}/\ideal{a}'$ is $\sigma$--torsion.
\item
$\ideal{a}\subseteq{A}$ is \textbf{totally $\sigma$--finitely generated} whenever there exists a finitely generated ideal $\ideal{a}'\subseteq\ideal{a}$ such that $\ideal{a}/\ideal{a}'$ is totally $\sigma$--torsion. \end{enumerate}

In the same way, for any ring $A$ we have two different notions of noetherian ring relative to $\sigma$:
\begin{enumerate}[(1)]\sepa
\item
$A$ is \textbf{$\sigma$--noetherian} if every ideal is $\sigma$--finitely generated.
\item
$A$ is \textbf{totally $\sigma$--noetherian} whenever every ideal is totally $\sigma$--finitely generated.
\end{enumerate}

\medskip
\begin{example}\label{ex:20210125b}
\begin{enumerate}[(1)]\sepa
\item
Every finitely generated ideal is totally $\sigma$--finitely generated and every totally $\sigma$--finitely generated ideal is $\sigma$--finitely generated.
\item
Let $S\subseteq{A}$ be a multiplicatively closed subset, an ideal $\ideal{a}\subseteq{A}$ is $S$--finite if, and only if, it is totally $\sigma_S$--finitely generated. The ring $A$ is $S$--noetherian if, and only if, $A$ is totally $\sigma_S$--noetherian
\end{enumerate}
\end{example}

We may dualize this notions, thus, if $A$ is a ring and $\sigma$ a hereditary torsion theory in $\rMod{A}$,
\begin{enumerate}[(1)]\sepa
\item
$A$ is \textbf{$\sigma$--artinian} if every decreasing chain of ideals is $\sigma$--stable.
\item
$A$ is \textbf{totally $\sigma$--artinian} if every decreasing chain of ideals is totally $\sigma$--stable.
\end{enumerate}

Being a decreasing chain of ideals $\ideal{a}_1\supseteq\ideal{a}_2\supseteq\cdots$ \textbf{$\sigma$--stable} whenever there exists an index $m$ such that $\ideal{a}_s\subseteq_\sigma\ideal{a}_m$, for every $s\geq{m}$, i.e., every $\ideal{a}_s$ is  \textbf{$\sigma$--dense} in $\ideal{a}_m$, or equivalently, for every $x\in\ideal{a}_m$ there exists $\ideal{h}\in\mathcal{L}(\sigma)$ such that $x\ideal{h}\subseteq\ideal{a}_s$ (observe that $\ideal{h}$ depends of $x$ and $s$). Otherwise, the decreasing chain of ideals is \textbf{totally $\sigma$--stable} whenever there exist an index $m$, and $\ideal{h}\in\mathcal{L}(\sigma)$ such that $\ideal{a}_m\ideal{h}\subseteq\ideal{a}_s$, for every $s\geq{m}$.

\begin{lemma}
For any ring $A$ we have:
\[
A\textrm{ is artinian}
\Rightarrow
A\textrm{ is totally $\sigma$--artinian}
\Rightarrow
A\textrm{ is $\sigma$--artinian.}
\]
\end{lemma}

The notions of $\sigma$--artinian (resp. totally $\sigma$--artinian) and $\sigma$--noetherian (resp. totally $\sigma$-noetherian) ring can be extended to $A$--modules in an easy way.

Trivial examples of totally $\sigma$--artinian modules are the totaly $\sigma$--torsion modules. Also every artinian module is totally $\sigma$--artinian for every hereditary torsion theory $\sigma$.

These two notions of torsion, and the derived notions from them, are completely different in its behaviour and its categorical properties. For instance, due to the definition, for any $A$--module $M$ there exists a maximum submodule belonging to $\mathcal{T}_\sigma$, the submodule: $\sigma{M}$, and it satisfies $M/\sigma{M}\in\mathcal{F}_\sigma$. In the totally $\sigma$--torsion case we cannot assure the existence of a maximal totally $\sigma$--torsion submodule. The existence of a maximum $\sigma$--torsion submodule allows us to build new concepts relative to $\sigma$ as lattices, closure operator and localization; concepts that we do not have in the totally $\sigma$--torsion case. For instance, the ring $A$ is $\sigma$--artinian if, and only if, the lattice $C(A,\sigma)=\{\ideal{a}\mid\;A/\ideal{a}\in\mathcal{F}_\sigma\}$ is an artinian lattice.
Nevertheless, the totally $\sigma$--torsion case allows us to study arithmetic properties of rings and modules which are hidden with that use of $\sigma$--torsion, and these properties are those which we are interested in studying.

As we point out before, the $\sigma$--torsion allows us, for any $A$--module $M$, to define a lattice
\[
C(M,\sigma)=\{N\subseteq{M}\mid\;M/N\in\mathcal{F}_\sigma\},
\]
and in $\mathcal{L}(M)$, the lattice of all submodules of $M$, a closure operator $\Clt{\sigma}{M}{-}:\mathcal{L}(M)\longrightarrow{C(M,\sigma)}\subseteq\mathcal{L}(M)$, defined by the equation $\Clt{\sigma}{M}{N}/N=\sigma(M/N)$. The elements in $C(M,\sigma)$ are called the \textbf{$\sigma$--closed} submodules of $M$, and the lattice operations in $C(M,\sigma)$, for any $N_1,N_2\in{C(M,\sigma)}$, are defined by
\[
\begin{array}{ll}
N_1\wedge{N_2}=N_1\cap{N_2},\\
N_1\vee{N_2}=\Clt{\sigma}{M}{N_1+N_2}.
\end{array}
\]
Dually, the submodules $N\subseteq{M}$ such that $M/N\in\mathcal{T}_\sigma$ are called \textbf{$\sigma$--dense} submodules. The set of all $\sigma$--dense submodules of $M$ is represented by $\mathcal{L}(M,\sigma)$, $\mathcal{L}(\sigma)$ in the case in which $M=A$.

In the following, we assume $A$ is a ring, $\rMod{A}$ is the category of $A$--modules and $\sigma$ is a hereditary torsion theory on $\rMod{A}$. Modules are represented by Latin letters: $M,N,N_1,\ldots$, and ideals by Gothics letters: $\ideal{a},\ideal{b},\ideal{b}_1,\ldots$ Different hereditary torsion theories will be represented by Greek letters: $\sigma,\tau,\sigma_1,\ldots$, and induced hereditary torsion theories by adorned Greek letters: $\sigma',\overline{\tau},\ldots$

In order to establish equivalent condition to (totally) $\sigma$--artinian modules, we introduce the definition of finitely cogenerated $A$--module.

\begin{enumerate}[(1)]\sepa
\item
An $A$--module $M$ is \textbf{finitely cogenerated} if for any family of submodules $\{N_i\mid\;i\in{I}\}$ such that $\cap_{i\in{I}}N_i=0$, there exists a finite subset $J\in{I}$ such that $\cap_{j\in{J}}N_j=0$.
\item
In the same way, in \cite{GOLAN:1986} the author uses the notion of $\sigma$--finitely cogenerated modules; an $A$--module $M$ is \textbf{$\sigma$--finitely cogenerated} if for any family of submodules $\{N_i\mid\;i\in{I}\}$ such that $\cap_{i\in{I}}N_i$ is $\sigma$--torsion there exists a finite subset $J\subseteq{I}$ such that $\cap_{j\in{J}}N_j$ is $\sigma$--torsion.
\item
In our case for totally $\sigma$--torsion, we define an $A$--module to be \textbf{totally $\sigma$--finitely cogenerated} whenever for every family of submodules $\{N_i\mid\;i\in{I}\}$ such that $\cap_{i\in{I}}N_i$ is totally $\sigma$--torsion there exists a finite subset $J\subseteq{I}$ such that $\cap_{j\in{J}}N_j$ is totally $\sigma$--torsion, i.e., there exists $\ideal{h}\in\mathcal{L}(\sigma)$ such that $(\cap_{j\in{J}}N_j)\ideal{h}=0$.
\end{enumerate}

\begin{theorem}
Let $A$ be a ring and $\sigma$ a hereditary torsion theory in $\rMod{A}$, for any $A$--module $M$ the following statements are equivalent:
\begin{enumerate}[(a)]\sepa
\item
$M$ is totally $\sigma$--artinian.
\item
Every quotient $M/N$ of $M$ is totally $\sigma$--finitely cogenerated.
\end{enumerate}
\end{theorem}
\begin{proof}
(a) $\Rightarrow$ (b). %
Let $\{N_i/N\mid\;i\in{I}\}$ be a family of submodules of $M/N$ such that $\cap_{i\in{I}}(N_i/N)$ is totally $\sigma$--torsion. If $H/N=(\cap_{i\in{I}}N_i)/N=\cap_{i\in{I}}(N_i/N)$, then $H/N$ is totally $\sigma$--torsion and $\cap_{i\in{I}}(N_i/H)=0$. We have a family $\{H_i=N_i/H\mid\;i\in{I}\}$ of submodules of $M/H$ such that $\cap_{i\in{I}}H_i=0$. By the hypothesis, $M/H$ is $\sigma$--artinian, so there are maximal elements in the set
\[
\Gamma=\{\cap_{j\in{J}}H_j\mid\;J\subseteq{I}\textrm{ is finite}\}.
\]
Let $\cap_{j\in{J}}H_j\in\Gamma$ be a minimal element in $\Gamma$. There exists $\ideal{h}\in\mathcal{L}(\sigma)$ such that for any $i\in{I\setminus{J}}$, we have
\[
(\cap_{j\in{J}}H_j)\ideal{h}\subseteq(\cap_{j\in{J}}H_j)\cap{N_i}\subseteq\cap_{j\in{J}}H_j.
\]
In particular, $(\cap_{j\in{J}}H_j)\ideal{h}\subseteq\cap_{i\in{I}}H_i=0$, and $\cap_{j\in{J}}H_j$ is totally $\sigma$--torsion.
\par
(b) $\Rightarrow$ (a). %
Let $N_0\supseteq{N_1}\supseteq\cdots$ be a decreasing chain of submodules of $M$, and define $N=\cap_{n\in\mathbb{N}}N_n$. In $M/N$ the family $\{N_n/N\mid\;n\in\mathbb{N}\}$ satisfies $\cap_{n\in\mathbb{N}}(N_n/N)=0$, hence there exists $J\subseteq\mathbb{N}$, finite, and $\ideal{h}\in\mathcal{L}(\sigma)$ such that $(\cap_{j\in{J}}(N_j/N))\ideal{h}=0$, hence $(N_k/N)\ideal{h}=0$, being $k=\max(J)$, and satisfies $N_k\ideal{h}\subseteq{N}$. Therefore, the decreasing chain $\sigma$--stabilizes.
\end{proof}

Properties about the behaviour of totally $\sigma$--finitely cogenerated and $\sigma$--noetherian modules are collected in the following result.

\medskip
\begin{proposition}\label{pr:20190102}
\begin{enumerate}[(1)]\sepa
\item
Every submodule of a totally $\sigma$--finitely cogenerated $A$--module also is .
\item
For every submodule $N\subseteq{M}$, we have: $M$ is totally $\sigma$--artinian if, and only if, $N$ and $M/N$ are totally $\sigma$--artinian.
\item
Finite direct sums of totally $\sigma$--artinian modules also are.
\end{enumerate}
\end{proposition}

Also we can build up examples of totally $\sigma$--artinian rings in considering hereditary torsion theories $\sigma_1\leq\sigma_2$. Thus, we have the following lemma, whose proof is straightforward.

\begin{lemma}
Let $\sigma_1\leq\sigma_2$ be hereditary torsion theories in $\rMod{A}$, and $M$ an $A$--module. If $M$ is totally $\sigma_1$--artinian then $M$ is totally $\sigma_2$--artinian.
\end{lemma}

Regular elements have a particular behaviour with respect to totally $\sigma$--artinian rings.

\begin{lemma}
If $A$ is a totally $\sigma$--artinian ring, for any regular element $a\in{A}$, we have $aA\in\mathcal{L}(\sigma)$.
\end{lemma}
\begin{proof}
If $a\in{A}$ is regular, we consider the decreasing chain $(a)\supseteq(a^2)\supseteq\cdots$. By the hypothesis, there exist an index $m$ and $\ideal{h}\in\mathcal{L}(\sigma)$ such that $(a^m)\ideal{h}\subseteq(a^s)$ for every $s\geq{m}$. Thus, for every $h\in\ideal{h}$ there exists $x\in{A}$ such that $a^mh=a^{m+1}x$, hence $h=ax\in{aA}$, which means that $\ideal{h}\subseteq{aA}$, and $aA\in\mathcal{L}(\sigma)$.
\end{proof}

As a consequence, the case of an integral domain is well understood. See~\cite[Corollary~2.2]{Sevim/Tekir/Koc:2020}.

\begin{corollary}
Let $A=D$ be an integral domain, if $D$ is totally $\sigma$--artinian, then $\sigma=\sigma_{D\setminus\{0\}}$, the usual torsion theory on $D$.
\end{corollary}

In particular, we have the following conclusions:
\begin{enumerate}[(1)]\sepa
\item
If $\ideal{p}\subseteq{D}$ is a non--zero prime ideal of an integral domain $D$, and we consider the hereditary torsion theory $\sigma_{D\setminus\ideal{p}}$, then $D$ is never totally $\sigma_{D\setminus\ideal{p}}$--artinian.
\item\label{it:20200103}
For every integral domain $D$, the hereditary torsion theory $\sigma=\sigma_{D\setminus\{0\}}$ satisfies that $D$ is $\sigma$--artinian, but non necessarily $D$ is totally $\sigma$--artinian. Indeed, $D$ is $\sigma$--artinian because $D_\sigma$, the field of fractions of $D$, is artinian. Otherwise the following example shows that the converse non necessarily holds. Let $D=\mathbb{Q}[X_n\mid\;n\in\mathbb{N}]$, and $\ideal{a}_n=(X_0\cdots{X_n})$, for every $n\in\mathbb{N}$; the decreasing chain $\ideal{a}_0\supseteq\ideal{a}_1\supseteq\cdots$ satisfies that there is not $s\in{D\setminus\{0\}}$ neither $m\in\mathbb{N}$ such that $\ideal{a}_ms\subseteq\ideal{a}_s$, for every $s\geq{m}$.
\end{enumerate}

This example raises the following problem:

\medskip
\begin{problem}
Which properties are necessary to add to a $\sigma$--artinian ring to be a totally $\sigma$--artinian ring?
\end{problem}

We refer to Theorem~\eqref{th:20200103} below.

\begin{corollary}
Let $A$ be a totally $\sigma$--artinian ring, and $\Reg(A)$ be the set of all regular elements of $A$, then $\sigma_{\Reg(A)}\leq\sigma$.
\end{corollary}
\begin{proof}
It is a consequence of the well known fact that $\Reg(A)$ is a saturated multiplicatively closed set.
\end{proof}

Let $A$ be a ring and $T=T(A)$ the \textbf{total ring of fractions} of $A$, i.e., the localization of $A$ at $\Reg(A)$, the multiplicatively closed set of all regular elements, i.e., $T=A_{\sigma_{\Reg(A)}}$. The above example in \eqref{it:20200103} shows that non necessarily $A$ must be totally $\sigma_{\Reg(A)}$--artinian, although it is $\sigma_{\Reg(A)}$--artinian.

We said that an ideal $\ideal{a}\subseteq{A}$ is \textbf{invertible} whenever $\ideal{a}(A:\ideal{a})=A$, being $(A:\ideal{a})=\{x\in{T}\mid\;\ideal{a}x\subseteq{A}\}$.

\begin{corollary}
If $A$ is a totally $\sigma$--artinian ring, every invertible ideal $\ideal{a}$ belongs to $\mathcal{L}(\sigma)$
\end{corollary}
\begin{proof}
Let $\ideal{a}\subseteq{A}$ be an invertible ideal, we consider the decreasing chain $\ideal{a}\supseteq\ideal{a}^2\supseteq\cdots$. By the hypothesis, there exist an index $m$ and $\ideal{h}\in\mathcal{L}(\sigma)$ such that $\ideal{a}^m\ideal{h}\subseteq\ideal{a}^s$ for every $s\geq{m}$. In particular, $\ideal{a}^m\ideal{h}\subseteq\ideal{a}^{m+1}$, hence $\ideal{h}\subseteq\ideal{a}$, and $\ideal{a}\in\mathcal{L}(\sigma)$.
\end{proof}

\begin{example}
Since invertible ideals are finitely generated ideals, they generate a here\-ditary torsion theory, that we name $\sigma_{\textrm{inv}}$, see \eqref{it:ex:20200103-2} in Example~\eqref{ex:20200103}. If $A$ is a totally $\sigma$--artinian ring, non necessarily $A$ is totally $\sigma_{\textrm{inv}}$--artinian.
\par
Indeed, we can consider the ring $A=\prod_\mathbb{N}\mathbb{Z}_2$, and the prime ideal $\ideal{p}=\prod_{n\geq1}\mathbb{Z}_2$. We know that $A$ is totally $\sAp$--artinian. Since $A$ is a total ring of fractions, every non regular element is invertible, hence $T=T(A)=A$, and the only invertible ideal is the proper $A$, hence $\sigma_{\Reg(A)}=0$. If $A$ were totally $\sigma_{\Reg(A)}$--artinian then $A$ must be exactly artinian, but obviously $A$ is not artinian.
\end{example}

In general, if $A$ is a totally $\sigma$--artinian ring, we have one more property of the hereditary torsion theory $\sigma$.

\begin{proposition}
If $A$ is a totally $\sigma$--artinian ring, the hereditary torsion theory $\sigma$ is of finite type.
\end{proposition}
\begin{proof}
Since $A$ is totally $\sigma$--artinian, it is $\sigma$--artinian and, by Hopkins' Theorem, $\sigma$--noetherian, hence $\sigma$ is of finite type.
\end{proof}

We are interested in proving stronger results: if $A$ is totally $\sigma$--artinian, then $A$ is totally $\sigma$--noetherian.

\section{Scalar extensions}

Let $f:A\longrightarrow{B}$ be a ring map. For every hereditary torsion theory $\sigma$ in $\rMod{A}$ we may define a new hereditary torsion theory $f(\sigma)$ in $\rMod{B}$ being
\begin{itemize}\sepa
\item
$\mathcal{L}(f(\sigma))=\{\ideal{b}\subseteq{B}\mid\;f^{-1}(\ideal{b})\in\mathcal{L}(\sigma)\}$,
\item
$\mathcal{T}_{f(\sigma)}=\{M_B\mid\;M_A\in\mathcal{T}_\sigma\}$,
\item
$\mathcal{F}_{f(\sigma)}=\{M_B\mid\;M_A\in\mathcal{F}_\sigma\}$.
\end{itemize}
In addition, sometimes, we shall impose the condition that every ideal of $B$ is an extension of an ideal of $A$, i.e., for every ideal $\ideal{b}\subseteq{B}$, there exists an ideal $\ideal{a}\subseteq{A}$ such that $\ideal{b}=f(\ideal{a})B$. With this condition, we have that the Gabriel filter $\mathcal{L}(f(\sigma))$ can be described also as
\[
\mathcal{L}(f(\sigma))=\{f(\ideal{a})B\mid\;\ideal{a}\in\mathcal{L}(\sigma)\}.
\]

\begin{lemma}
Let $f:A\longrightarrow{B}$ be a ring map such that every ideal of $B$ is an extension of an ideal of $A$, and let $\sigma$ be a hereditary torsion theory in $\rMod{A}$ such that $A$ is totally $\sigma$--artinian, then $B$ is totally $\sigma$--artinian.
\end{lemma}
\begin{proof}
Let $\ideal{b}_1\supseteq\ideal{b}_2\supseteq\cdots$ be a decreasing chain of ideals of $B$, and let $\ideal{a}_i\subseteq{A}$ be an ideal such that $\ideal{b}_i=f(\ideal{a}_i)B$; we can obtain a decreasing chain $\ideal{a}_1\supseteq\ideal{a}_2\supseteq\cdots$ of ideals $A$. By the hypothesis, there exist an index $m$ and $\ideal{h}\in\mathcal{L}(\sigma)$ such that $\ideal{a}_m\ideal{h}\subseteq\ideal{a}_s$ for every $s\geq{m}$. In consequence, $\ideal{b}_m\ideal{h}B\subseteq\ideal{b}_s$ for every $s\geq{m}$, and $B$ is totally $f(\sigma)$--artinian.
\end{proof}

\begin{corollary}
Let $A$ be a totally $\sigma$--artinian ring, then we have:
\begin{enumerate}[(1)]\sepa
\item
If for any ideal $\ideal{a}\subseteq{A}$ we consider the canonical projection $p:A\longrightarrow{A/\ideal{a}}$, then $A/\ideal{a}$ is totally $p(\sigma)$--artinian.
\item
If for any multiplicatively closed subset $\Sigma\subseteq{A}$ we consider the canonical map $f:A\longrightarrow\Sigma^{-1}A$, then $\Sigma^{-1}A$ is totally $f(\sigma)$--artinian.
\end{enumerate}
\end{corollary}

\begin{corollary}
Let $\ideal{a}\subseteq{A}$ be an ideal, and $p:A\longrightarrow{A/\ideal{a}}$ the canonical projection. The following statements are equivalent:
\begin{enumerate}[(a)]\sepa
\item
$A$ is totally $\sigma$--artinian.
\item
$\ideal{a}$ is totally $\sigma$--artinian and $A/\ideal{a}$ is totally $\sigma$--artinian (equivalently, it is totally $p(\sigma)$--artinian).
\end{enumerate}
\end{corollary}

And, as a consequence of Proposition\eqref{pr:20190102}, we have:

\begin{corollary}
Let $A$ be a totally $\sigma$--artinian ring, then every finitely generated $A$--module is totally $\sigma$--artinian.
\end{corollary}

\section{The minimal condition}\label{se:20190102}

Let $M$ be an $A$--module, after \cite{Sevim/Tekir/Koc:2020}, we establish the following definitions:

\begin{enumerate}[(1)]\sepa
\item
Let $\mathcal{N}\subseteq\mathcal{L}(M)$ be a family of submodules of $M$. An element $N\in\mathcal{N}$ is \textbf{$\sigma$--minimal} if there exists $\ideal{h}\in\mathcal{L}(\sigma)$ such that for every $H\in\mathcal{N}$ such that $H\subseteq{N}$ we have $N\ideal{h}\subseteq{H}$.
\item
The $A$--module $M$ satisfies the \textbf{$\sigma$-MIN condition} if every nonempty family of submodules of $M$ has $\sigma$-minimal elements.
\item
A family $\mathcal{N}$ of submodules of $M$ is \textbf{$\sigma$--lower closed} if for every submodule $H\subseteq{M}$ such that there exist $N\in\mathcal{N}$ and $\ideal{h}\in\mathcal{L}(\sigma)$ satisfying $N\ideal{h}\subseteq{H}$, either equivalently $N\subseteq(H:\ideal{h})$ or equivalently $(H:N)\in\mathcal{L}(\sigma)$, we have $H\in\mathcal{N}$.
\end{enumerate}

We have the following characterization of totally $\sigma$--artinian modules.

\begin{proposition}
Let $M$ be an $A$--module, the following statements are equivalent:
\begin{enumerate}[(a)]\sepa
\item
$M$ is totally $\sigma$--artinian.
\item
Every nonempty $\sigma$--lower closed family of submodules of $M$ has minimal elements.
\item
Every nonempty family of submodules of $M$ has $\sigma$--minimal elements.
\end{enumerate}
\end{proposition}

If we have $\Sigma\subseteq{A}$ a multiplicatively closed subset of $A$ and $\sigma=\sigma_\Sigma$, this proposition is Theorem 2.1 in \cite{Sevim/Tekir/Koc:2020}.

Let $\sigma$ be a hereditary torsion theory in $\rMod{A}$; an $A$--module $M$ is
\begin{enumerate}[(1)]\sepa
\item
\textbf{$\sigma$--finitely cogenerated} if for any family of submodules $\{N_i\mid\;i\in{I}\}$ of $M$ such that $\cap_{i\in{I}}N_i$ is $\sigma$--torsion, there exists a finite subset $J\subseteq{I}$ such that $\cap_{j\in{J}}N_j$ is $\sigma$--torsion.
\item
\textbf{totally $\sigma$--finitely cogenerated} if for any family of submodules $\{N_i\mid\;i\in{I}\}$ of $M$ such that $\cap_{i\in{I}}N_i$ is totally $\sigma$--torsion, there exists a finite subset $J\subseteq{I}$ such that $\cap_{j\in{J}}N_j$ is totally $\sigma$--torsion.
\item
\textbf{strongly totally $\sigma$--finitely cogenerated} if for any family of submodules $\{N_i\mid\;i\in{I}\}$ of $M$ such that $\cap_{i\in{I}}N_i=0$, there exists a finite subset $J\subseteq{I}$ such that $\cap_{j\in{J}}N_j$ is totally $\sigma$--torsion.
\end{enumerate}

We are mainly interested in modules $M$ such that every quotient $M/N$ is $\sigma$--finitely cogenerated (resp. totally $\sigma$--finitely cogenerated). For that reason we weaken the condition $\cap_{i\in{I}}N_i$ is $\sigma$--torsion (resp. totally $\sigma$--torsion) to simply consider that $\cap_{i\in{I}}N_i=0$, obtaining in this way the strongly totally $\sigma$--cogenerated modules. These modules will be also useful in order to compare hereditary torsion theories because if $\sigma_1\leq\sigma_2$, an $A$--module $M$ may be totally $\sigma_1$--finitely cogenerated and non necessarily totally $\sigma_2$--finitely cogenerated.

Now we can give another characterization of totally $\sigma$--artinian modules in the following way.

\begin{theorem}\label{th:20200104}
Let $A$ be a ring and $\sigma$ a hereditary torsion theory in $\rMod{A}$, for any $A$--module $M$ the following statements are equivalent:
\begin{enumerate}[(a)]\sepa
\item\label{it:th:20200104-1}
$M$ is totally $\sigma$--artinian.
\item\label{it:th:20200104-2}
Every quotient $M/N$ of $M$ is strongly totally $\sigma$--finitely cogenerated.
\item\label{it:th:20200104-3}
Every quotient $M/N$ of $M$ is totally $\sigma$--finitely cogenerated.
\end{enumerate}
\end{theorem}
\begin{proof}
\eqref{it:th:20200104-1} $\Rightarrow$ \eqref{it:th:20200104-2}. %
Let $\{N_i/N\mid\;i\in{I}\}$ be a family of submodules of $M/N$ such that $\cap_{i\in{I}}(N_i/N)=0$, hence $\cap_{i\in{I}}N_i=N$. By the hypothesis, $M/N$ is totally $\sigma$--artinian, so there are maximal elements in the set
\[
\Gamma=\{\cap_{j\in{J}}N_j\mid\;J\subseteq{I}\textrm{ is finite}\}.
\]
Let $\cap_{j\in{J}}N_j\in\Gamma$ be a minimal element in $\Gamma$. There exists $\ideal{h}\in\mathcal{L}(\sigma)$ such that for any $i\in{I\setminus{J}}$, we have
\[
(\cap_{j\in{J}}N_j)\ideal{h}\subseteq(\cap_{j\in{J}}N_j)\cap{N_i}\subseteq\cap_{j\in{J}}N_j.
\]
In particular, $(\cap_{j\in{J}}N_j)\ideal{h}\subseteq\cap_{i\in{I}}N_i=0$, and $\cap_{j\in{J}}H_j$ is totally $\sigma$--torsion.
\par
\eqref{it:th:20200104-1} $\Rightarrow$ \eqref{it:th:20200104-3}. %
Let $\{N_i/N\mid\;i\in{I}\}$ be a family of submodules of $M/N$ such that $\cap_{i\in{I}}(N_i/N)$ is totally $\sigma$--torsion. If $H/N=(\cap_{i\in{I}}N_i)/N=\cap_{i\in{I}}(N_i/N)$, then $H/N$ is totally $\sigma$--torsion and $\cap_{i\in{I}}(N_i/H)=0$. We have a family $\{H_i=N_i/H\mid\;i\in{I}\}$ of submodules of $M/H$ such that $\cap_{i\in{I}}H_i=0$. By the hypothesis, $M/H$ is $\sigma$--artinian, so there are maximal elements in the set
\[
\Gamma=\{\cap_{j\in{J}}H_j\mid\;J\subseteq{I}\textrm{ is finite}\}.
\]
Let $\cap_{j\in{J}}H_j\in\Gamma$ be a minimal element in $\Gamma$. There exists $\ideal{h}\in\mathcal{L}(\sigma)$ such that for any $i\in{I\setminus{J}}$, we have
\[
(\cap_{j\in{J}}H_j)\ideal{h}\subseteq(\cap_{j\in{J}}H_j)\cap{H_i}\subseteq\cap_{j\in{J}}H_j.
\]
In particular, $(\cap_{j\in{J}}H_j)\ideal{h}\subseteq\cap_{i\in{I}}H_i=0$, and $\cap_{j\in{J}}H_j$ is totally $\sigma$--torsion. Therefore also $\cap_{j\in{J}}(N_j/H)$ and $\cap_{j\in{J}}(N_j/N)$ are totally $\sigma$--torsion
\par
\eqref{it:th:20200104-2} $\Rightarrow$ \eqref{it:th:20200104-1},
\eqref{it:th:20200104-3} $\Rightarrow$ \eqref{it:th:20200104-1}. %
Let $N_0\supseteq{N_1}\supseteq\cdots$ be a decreasing chain of submodules of $M$, and define $N=\cap_{n\in\mathbb{N}}N_n$. In $M/N$ the family $\{N_n/N\mid\;n\in\mathbb{N}\setminus\{0\}\}$ satisfies $\cap_{n\in\mathbb{N}}(N_n/N)=0$, hence there exist $J\subseteq\mathbb{N}$, finite, and $\ideal{h}\in\mathcal{L}(\sigma)$ such that $(\cap_{j\in{J}}(N_j/N))\ideal{h}=0$, hence $(N_k/N)\ideal{h}=0$, being $k=\max(J)$, and satisfies $N_k\ideal{h}\subseteq{N}$. Therefore the decreasing chain $\sigma$--stabilizes.
\end{proof}

\begin{lemma}\label{le:20191204}
Let $M$ be an $A$--module and $T\subseteq{M}$ be a totally $\sigma$--torsion submodule, the following statements are equivalent:
\begin{enumerate}[(a)]\sepa
\item
$M$ is totally $\sigma$--artinian.
\item
$M/T$ is totally $\sigma$--artinian.
\end{enumerate}
\end{lemma}
It is a direct consequence  of Proposition~\eqref{pr:20190102}, since every totally $\sigma$--torsion module is totally $\sigma$--artinian.

In the same line, we find that finitely cogenerated modules have their own characterization. The following is Theorem~3.4 in \cite{Sevim/Tekir/Koc:2020}.

\begin{theorem}
Let $M$ be an $A$--module, the following statements are equivalent:
\begin{enumerate}[(a)]\sepa
\item
$M$ is finitely cogenerated.
\item
$M$ is strongly totally $\sAp$--finitely cogenerated, for every $\ideal{p}\in\Spec(A)$.
\item
$M$ is strongly totally $\sAm$--finitely cogenerated, for every $\ideal{m}\in\Max(A)$.
\end{enumerate}
\end{theorem}
\begin{proof}
(a) $\Rightarrow$ (b) $\Rightarrow$ (c). %
They are obvious.
\par
(c) $\Rightarrow$ (a). %
Let $\{N_i\mid\;i\in{I}\}$ be a family of submodules such that $\cap_{i\in{I}}N_i=0$, for every maximal ideal $\ideal{m}\subseteq{A}$ there exist a finite subset $J_\ideal{m}\subseteq{I}$ and $s_\ideal{m}\in{A\setminus\ideal{m}}$ such that $(\cap_{j\in{J_\ideal{m}}}N_j)s_\ideal{m}=0$. If $\ideal{h}=\langle{s_\ideal{m}\mid\;\ideal{m}\in\Max(A)}\rangle\neq{A}$ there exists a maximal ideal $\ideal{m}$ such that $\ideal{h}\subseteq\ideal{m}$, which is a contradiction. Therefore, $\ideal{h}=A$, and there are maximal ideals $\ideal{m}_1,\ldots,\ideal{m}_t$ such that $\langle{s_{\ideal{m}_1},\ldots,s_{\ideal{m}_t}}\rangle=A$. We define $J=\cup_{i=1}^tJ_{\ideal{m}_i}\subseteq{I}$, which is finite, and satisfies $(\cap_{j\in{J}}N_j)s_{\ideal{m}_i}=0$, for every $i=1,\ldots,t$. Hence $\cap_{j\in{J}}N_j=(\cap_{j\in{J}}N_j)\langle{s_{\ideal{m}_1},\ldots,s_{\ideal{m}_t}}\rangle=0$, and $M$ is finitely cogenerated.
\end{proof}

Also, if $A=D$ is an integral domain strongly totally $\sigma$--finitely cogenerated then $D$ is a field, whenever $\sigma\neq0$.

\begin{proposition}
If $D$ is a $\sigma$--torsionfree strongly totally $\sigma$--finitely cogenerated integral domain, then $D$ is a field. The converse always holds.
\end{proposition}
\begin{proof}
We claim $0\subseteq{D}$ is strongly prime,  see \cite{Gottlieb:2017}. Indeed, let $\{\ideal{a}_i\mid\;i\in{I}\}$ be a family of ideals such that $\cap_{i\in{I}}\ideal{a}_i=0$. By the hypothesis, there exist $J\subseteq{I}$, finite, and $\ideal{h}\in\mathcal{L}(\sigma)$ such that $(\cap_{j\in{J}}\ideal{a}_j)\ideal{h}=0$. Since $A$ is $\sigma$--torsionfree, then $\cap_{j\in{J}}\ideal{a}_j=0$. Otherwise, since $0\subseteq{D}$ is prime, there exists $j\in{J}$ such that $\ideal{a}_j=0$.
\par
In consequence, the intersection of all non--zero ideals is non--zero, and $D$ contains a minimum non-zero ideal, say $\ideal{a}=aD$. For any $0\neq{x}\in{D}$ we have $ax\neq0$, and $aD\subseteq{axD}$, this means that there exists $y\in{D}$ such that $a=axy$, and $1=xy$, hence $x$ is invertible.
\end{proof}

We put the condition that $D$ is $\sigma$--torsionfree only to avoid the trivial case in which $\sigma=0$.

\section{Study through prime ideals}

Let $\ideal{p}\subseteq{A}$ be a prime ideal, we consider $\sAp$, the hereditary torsion theory cogenerated by $A/\ideal{p}$, or equivalently, the hereditary torsion theory generated by the multiplicatively subset $A\setminus\ideal{p}$. For every torsion theory $\sigma$ we associate the following sets of ideals:
\begin{enumerate}[(1)]\sepa
\item
$\mathcal{L}(\sigma)$, the Gabriel filter of $\sigma$.
\item
$\mathcal{Z}(\sigma)=\mathcal{L}(\sigma)\cap\Spec(A)$. In particular, if $\ideal{p}\subseteq\ideal{q}$ are prime ideals and $\ideal{p}\in\mathcal{Z}(\sigma)$, then $\ideal{q}\in\mathcal{Z}(\sigma)$.
\item
$C(A,\sigma)=\{\ideal{a}\mid\;A/\ideal{a}\in\mathcal{F}_\sigma\}$.
\item
$\mathcal{K}(\sigma)=C(A,\sigma)\cap\Spec(A)$, it is the complement of $\mathcal{Z}(\sigma)$ in $\Spec(A)$. In particular, if $\ideal{p}\subseteq\ideal{q}$ are prime ideals and $\ideal{q}\in\mathcal{Z}(\sigma)$, then $\ideal{p}\in\mathcal{Z}(\sigma)$.
\item
$\mathcal{C}(\sigma)=\Max\mathcal{K}(\sigma)$.
\end{enumerate}
If $\sigma$ is of finite type, then $\sigma=\wedge\{\sAp\mid\;\ideal{p}\in\mathcal{K}(\sigma)\}$. Otherwise, $\sigma=\wedge\{\sAp\mid\;\ideal{p}\in\mathcal{C}(\sigma)\}$ whenever $A$ is $\sigma$--noetherian, because $\sigma_{A\setminus\ideal{q}}\leq\sAp$ whenever $\ideal{p}\subseteq\ideal{q}$, for any prime ideals $\ideal{p},\ideal{q}$.

\begin{proposition}\label{pr:20191026}
Let $\ideal{p}\subseteq{A}$ be a prime ideal. If $A$ is totally $\sAp$--artinian, then $\ideal{p}\subseteq{A}$ is a minimal prime ideal.
\end{proposition}
\begin{proof}
Let $\ideal{q}\subsetneqq\ideal{p}$ be prime ideals such that $A$ is totally $\sAp$--artinian.  Taking the quotient by the ideal $\ideal{q}$, $p:A\longrightarrow{A/\ideal{q}}$, we have that $A/\ideal{q}$ is a totally $p(\sAp)$--artinian domain, hence $p(\sAp)$ is the usual hereditary torsion theory in a domain, i.e., $\mathcal{L}(p(\sAp))$ contains only the non--zero ideals of $A/\ideal{q}$. Therefore, $\ideal{q}\notin\mathcal{L}(\sAp)$, which is a contradiction.
\end{proof}

\begin{corollary}\label{co:20191118b}
Let $A$ be a totally $\sigma$--artinian ring, every prime ideal $\ideal{p}\in\mathcal{K}(\sigma)$ is a minimal prime ideal. In consequence, $\mathcal{K}(\sigma)=\mathcal{C}(\sigma)$, i.e., every prime ideal in $\mathcal{K}(\sigma)$ is maximal in $\mathcal{K}(\sigma)$.
\end{corollary}
\begin{proof}
Let $\ideal{p}\in\mathcal{K}(\sigma)$, then $\sigma\leq\sAp$, and $A$ is $\sAp$--artinian. Therefore, $\ideal{p}$ is a minimal prime ideal, hence maximal in $\mathcal{K}(\sigma)$.
\end{proof}

For any multiplicatively closed subset $\Sigma\subseteq{A}$, and $\sigma=\sigma_\Sigma$, we obtain Proposition~2.5 in \cite{Sevim/Tekir/Koc:2020}.

Since every totally $\sigma$--artinian ring is $\sigma$--artinian, we establish the next result for $\sigma$--artinian rings.

\begin{corollary}\label{co:20191118}
Let $A$ be a $\sigma$--artinian ring, the following statements hold:
\begin{enumerate}[(1)]\sepa
\item\label{it:co:20191118-1}
$C(A,\sigma)$ is artinian, and the converse also holds.
\item\label{it:co:20191118-4}
$A$ is $\sigma$--noetherian. In particular, $\sigma$ is of finite type.
\item\label{it:co:20191118-2}
$\mathcal{K}(\sigma)$ is a finite set, say $\mathcal{K}(\sigma)=\{\ideal{p}_1,\ldots,\ideal{p}_t\}$.
\item\label{it:co:20191118-3}
There exists a multiplicatively closed subset $\Sigma\subseteq{A}$ such that $\mathcal{L}(\sigma)=\{\ideal{a}\mid\;\ideal{a}\cap\Sigma\neq\varnothing\}$ and $A_\Sigma$ is artinian. The converse also holds.
\end{enumerate}
\end{corollary}
\begin{proof}
\eqref{it:co:20191118-1}. %
It is just the definition.
\par
\eqref{it:co:20191118-4}. %
It is Hopkins' Theorem. See \cite{GOLAN:1986}.
\par
\eqref{it:co:20191118-2}. %
If $\mathcal{K}(\sigma)$ is not finite, there exists a numerable family of prime ideals $\{\ideal{p}_n\mid\;n\in\mathbb{N}\}\subseteq\mathcal{K}(\sigma)$. Then we may build a decreasing chain of ideals $\ideal{p}_1\supseteq\ideal{p}_1\ideal{p}_2\supseteq\ideal{p}_1\ideal{p}_2\ideal{p}_3\supseteq\cdots$. By the hypothesis this chain stabilizes, and there exists an index $m$ such that $\Clt{\sigma}{A}{\ideal{p}_1\cdots\ideal{p}_m}=\Clt{\sigma}{A}{\ideal{p}_1\cdots\ideal{p}_m\ideal{p}_{m+1}}=\cdots$. Therefore, for any $x\in\ideal{p}_1\cdots\ideal{p}_m$ there exists $\ideal{h}\in\mathcal{L}(\sigma)$ such that $x\ideal{h}\subseteq\ideal{p}_1\cdots\ideal{p}_m\ideal{p}_{m+1}\subseteq\ideal{p}_{m+1}$. Since $\ideal{h}\nsubseteq\ideal{p}_{m+1}$, we have $x\in\ideal{p}_{m+1}$. In consequence, $\ideal{p}_1\cdots\ideal{p}_m\subseteq\ideal{p}_{m+1}$, and there exists an index $i$ such that $\ideal{p}_i\subseteq\ideal{p}_{m+1}$, which is a contradiction.
\par
\eqref{it:co:20191118-3}. %
See \cite[Corollary 7.24]{ALBU/NASTASESCU:1984}.
\end{proof}

The following result appears as Theorem~2.2 in \cite{Sevim/Tekir/Koc:2020}.

\begin{theorem}
Let $A$ be a ring, the following statements are equivalent:
\begin{enumerate}[(a)]\sepa
\item
$A$ is artinian.
\item
$A$ is totally $\sAp$--artinian, for every $\ideal{p}\in\Spec{A}$.
\item
$A$ is totally $\sAm$--artinian, for every $\ideal{m}\in\Max(A)$.
\end{enumerate}
\end{theorem}
\begin{proof}
(a) $\Rightarrow$ (b) $\Rightarrow$ (c). It is evident. %
\par
(c) $\Rightarrow$ (a). %
Let $\ideal{a}_1\supseteq\ideal{a}_2\supseteq\cdots$ be a decreasing chain of ideals. For every maximal ideal $\ideal{m}$ there exist $m_\ideal{m}$ and $s_\ideal{m}\in{A\setminus\ideal{m}}$ such that $\ideal{a}_{m_\ideal{m}}\subseteq\ideal{a}_{s}$ for every $s\geq{m_\ideal{m}}$. Let $G=\{s_\ideal{m}\mid\;\ideal{m}\in\Max(A)\}$. If $(G)\neq{A}$, there exists a maximal ideal $\ideal{m}$ such that $(G)\subseteq\ideal{m}$, which is a contradiction. Thus we have $(G)=A$, and there exist $s_1,\ldots,s_t\in{G}$ such that $\sum_{i=1}^ts_iA=A$. Let $s_i=s_{\ideal{m}_i}$, and $m_i=m_{\ideal{m}_i}$ for every $i=1,\ldots,t$. If $m=\max\{m_1,\ldots,m_t\}$ we have $\ideal{a}_m\subseteq\ideal{a}_{m_i}$, then
\[
\ideal{a}_ms_i\subseteq\ideal{a}_{m_i}s_i\subseteq\ideal{a}_s,
\]
for every $s\geq{m}$. In consequence, $\ideal{a}_m=\ideal{a}_m\sum_{i=1}^ts_iA\subseteq\ideal{a}_s$ for every $s\geq{m}$, and the chain stabilizes.
\end{proof}

We may study this result in order to characterize totally $\sigma$--artinian rings.

\begin{theorem}\label{th:20200103}
Let $A$ be a ring and $\sigma$ be a finite type hereditary torsion theory. The following statements are equivalent:
\begin{enumerate}[(a)]\sepa
\item
$A$ is totally $\sigma$--artinian.
\item
$\mathcal{K}(\sigma)=\mathcal{C}(\sigma)$ is finite and $A$ is totally $\sAp$--artinian for every prime ideal $\ideal{p}\in\mathcal{C}(\sigma)$.
\end{enumerate}
\end{theorem}
\begin{proof}
  (a) $\Rightarrow$ (b). %
  See Corollaries \eqref{co:20191118b} and \eqref{co:20191118}.
  \par
  (b) $\Rightarrow$ (a). %
  Let $\ideal{a}_1\supseteq\ideal{a}_2\supseteq\cdots$ be a chain of ideals. For every $\ideal{p}\in\mathcal{C}(\sigma)$ there exists $m_\ideal{p}\in\mathbb{N}$ and $\ideal{h}_\ideal{p}\in\mathcal{L}(\sAp)$ such that $\ideal{a}_{m_\ideal{p}}\ideal{h}_\ideal{p}\subseteq\ideal{a}_s$ for every $s\geq{m_\ideal{p}}$. If we take $m=\max\{m_\ideal{p}\mid\;\ideal{p}\in\mathcal{C}(\sigma)\}$, and $\ideal{h}=\prod\{\ideal{h}_\ideal{p}\mid\;\ideal{p}\in\mathcal{C}(\sigma)\}$, then $\ideal{m}\ideal{h}\subseteq\ideal{a}_s$ for every $s\geq{m}$, which shows that $A$ is totally $\sigma$--artinian.
\end{proof}

\section{Simple and maximal modules}

In order to show that every totally $\sigma$--artinian ring is totally $\sigma$--noetherian we need to study simple modules and maximal ideal relative to the hereditary torsion theory $\sigma$. We shall use the preceding studies of minimal and maximal elements as appears in sections\eqref{se:20190102} and \eqref{se:20200102}, respectively.

\subsection*{The example of $\sigma$--torsion. The classical theory}

Let $M$ be an $A$--module. We have that $M$ is \textbf{$\sigma$--artinian} if, and only if, the family $C(M,\sigma)$ of $\sigma$--closed submodules $\mathcal{M}$ satisfies the decreasing chain condition, or equivalently the minimal condition, which are also equivalent to the condition that for every decreasing chain $N_1\supseteq{N_2}\supseteq\cdots$ of submodules there exists $m\in\mathbb{N}$ such that $\Clt{\sigma}{M}{N_m}=\Clt{\sigma}{M}{N_s}$, for every $s\geq{m}$.

A submodule $N\subseteq{M}$ is \textbf{$\sigma$--minimal} if $\Clt{\sigma}{M}{N}\subseteq{M}$ is a minimal element in $C(M,\sigma)\setminus\{\sigma{M}\}$, or equivalently if $C(N,\sigma)=\{\sigma{N},N\}$. This means that $N$ is not $\sigma$--torsion, and for every submodule $L\subseteq{N}$ we have either $L$ is $\sigma$--torsion or $L$ is not $\sigma$--torsion, and in this case $L\subseteq_\sigma{N}$, it is $\sigma$--dense. Observe that if $M$ is not $\sigma$--torsion, a submodule $N\subseteq{M}$ is $\sigma$--minimal if, and only if, $N$ is a minimal in the family $\mathcal{M}=\{N\subseteq{M}\mid\;N\textrm{ is not $\sigma$--torsion}\}$.

An $A$--module $M$ is \textbf{$\sigma$--simple} if $C(M,\sigma)=\{\sigma{M},M\}$. If, in addition, $M$ is $\sigma$--torsionfree, we name $M$ a \textbf{$\sigma$--cocritical} $A$--module.

We may dualize $\sigma$--artinian to obtain $\sigma$--noetherian modules. In the case of $\sigma$--maximal submodules, we have that $N\subseteq{M}$ is \textbf{$\sigma$--maximal} if $M/N$ is not $\sigma$--torsion, and for every submodule $N\subseteq{L}\subseteq{M}$ such that $M/L$ is not $\sigma$--torsion we have $N\subseteq_\sigma{L}$, which is equivalent to say that $M/N$ is a $\sigma$--simple $A$--module. Observe that if $M$ is not $\sigma$--torsion, a submodule $N\subseteq{M}$ is $\sigma$--maximal if, and only if, $N$ is a maximal the family $\mathcal{M}=\{N\subseteq{M}\mid\;M/N\textrm{ is not $\sigma$--torsion}\}$.

A submodule $N\subseteq{M}$ is called \textbf{$\sigma$--critical} if it is $\sigma$--maximal and $M/N$ is $\sigma$--torsionfree.

\subsection*{The example of totally $\sigma$--torsion. Simple modules}

When we study modules and the totally $\sigma$--torsion we need to change the paradigm. Thus, let $M$ be a totally $\sigma$--artinian $A$--module, and consider the family of submodules
\[
\mathcal{M}=\{N\subseteq{M}\mid\;N\textrm{ is not totally $\sigma$--torsion}\}.
\]
$\mathcal{M}$ is not empty whenever $M$ is not totally $\sigma$--torsion. If $M$ is totally $\sigma$--artinian, there exists a $\sigma$--minimal element in $\mathcal{M}$, say $N$, that satisfies:
\begin{enumerate}[(1)]\sepa
\item
$N$ is not totally $\sigma$--torsion,
\item
There exists $\ideal{h}\in\mathcal{L}(\sigma)$ such that for every $H\subseteq{N}$, which is not totally $\sigma$--torsion, we have $N\ideal{h}\subseteq{H}$.
\end{enumerate}

In general, for any $A$--module $M$, a submodule $N$ satisfying (1) and (2) is called a \textbf{totally $\sigma$--minimal submodule} of $M$. An $A$--module $M$ is called \textbf{totally $\sigma$--simple} whenever $M$ is a $\sigma$--minimal element of $\mathcal{M}$, i.e.,
\begin{enumerate}[(1)]\sepa
\item
$M$ is not totally $\sigma$--torsion and
\item
there exists $\ideal{h}\in\mathcal{L}(\sigma)$ such that for every not totally $\sigma$--torsion submodule $H\subseteq{M}$ we have $M\ideal{h}\subseteq{H}$.
\end{enumerate}

Let $M$ be a totally $\sigma$--simple $A$--module with \textbf{companion ideal} $\ideal{h}\in\mathcal{L}(\sigma)$, i.e., $\ideal{h}$ satisfies that for every $H\subseteq{M}$ which is non totally $\sigma$--torsion we have $M\ideal{h}\subseteq{H}$.

Observe that we have:

\begin{proposition}\label{pr:20191130}
Let $M$ be a totally $\sigma$--simple $A$--module with companion ideal $\ideal{h}\in\mathcal{L}(\sigma)$, the following statements hold:
\begin{enumerate}[(1)]\sepa
\item
$M\ideal{h}$ is not totally $\sigma$--torsion, hence $M\ideal{h}\subseteq{M}$ is the minimum of all not totally $\sigma$--torsion submodules of $M$. In particular, every not totally $\sigma$--torsion of $M$ is totally $\sigma$--simple.
\item
$M\ideal{h}$ is also totally $\sigma$--simple, and every proper submodule if totally $\sigma$--torsion.
\item
For any ideal $\ideal{h}'\in\mathcal{L}(\sigma)$ we always have $M\ideal{h}\subseteq{M\ideal{h}'}$. In particular, if $\ideal{h}'\subseteq\ideal{h}$ then $M\ideal{h}=M\ideal{h}'$.
\item
If $\ideal{h}'\in\mathcal{L}(\sigma)$ is another ideal companion to $M$, then $M\ideal{h}'=M\ideal{h}$.
\item\label{it:pr:20191130}
Let $M'$ be a totally $\sigma$--simple $A$--module and $f:M\longrightarrow{M'}$ be a surjective map with kernel $K$, then $K$ is totally $\sigma$--torsion.
\end{enumerate}
\end{proposition}
\begin{proof}
(1) to (4) are straightforward.
\par
\eqref{it:pr:20191130}. %
If $K\subseteq{M}$ is not totally $\sigma$--torsion, and $\ideal{h}\in\mathcal{L}(\sigma)$ is the companion ideal of $M$, then $M\ideal{h}\subseteq{K}$, and we have $M'\ideal{h}=0$, which is a contradiction.
\end{proof}

If $M$ is totally $\sigma$--simple with companion ideal $\ideal{h}$, we call $M\ideal{h}$ the \textbf{core submodule} of $M$.

\begin{proposition}
If $M\subseteq{M'}$ satisfies that there exists $\ideal{h}_0\in\mathcal{L}(\sigma)$ such that $M'\ideal{h}_0\subseteq{M}$, then $M$ is totally $\sigma$--simple if, and only if, $M'$ is. In addition, the core of $M$ and $M'$ are equal.
\end{proposition}
\begin{proof}
If $M$ is totally $\sigma$--simple, $\ideal{h}\in\mathcal{L}(\sigma)$ is a companion ideal, and $H\subseteq{M'}$ is not totally $\sigma$--torsion, then $H\ideal{h}_0\subseteq{M}$ is not totally $\sigma$--torsion, hence $M\ideal{h}\subseteq{H\ideal{h}_0}\subseteq{H}$, and $M'\ideal{h}_0\ideal{h}\subseteq{H}$. Otherwise, if $M'$ is totally $\sigma$--simple, $\ideal{h}'\in\mathcal{L}(\sigma)$ is a companion ideal, and $H\subseteq{M}$ is not totally $\sigma$--torsion, since $H\subseteq{M'}$, we have $M\ideal{h}'\subseteq{M'\ideal{h}'}\subseteq{H}$, and $M$ is totally $\sigma$--simple.
\par
If $M\subseteq{M'}$ are totally $\sigma$--simple modules we may assume $\ideal{h}$ is the same companion ideal to $M$ and $M'$, then we have $M\ideal{h}=M'\ideal{h}$. Indeed, $M'\ideal{h}\subseteq{M}$, and $M'\ideal{h}=M'\ideal{h}\ideal{h}\subseteq{M\ideal{h}}\subseteq{M'\ideal{h}}$. The core of $M/L$ is a quotient of the core of $M$.
\end{proof}

\begin{proposition}
Let $L\subseteq{M}$ be a totally $\sigma$--torsion submodule, then $M$ is totally $\sigma$--simple if, and only if, $M/L$ is. In this case, the core of $M/L$ is a quotient of the core of $M$.
\end{proposition}
\begin{proof}
If $M$ is totally $\sigma$--simple with companion ideal $\ideal{h}\in\mathcal{L}(\sigma)$, and $H/L\subseteq{M/L}$ is a not totally $\sigma$--torsion submodule, then $H\subseteq{M}$ is not totally $\sigma$--torsion, hence $M\ideal{h}\subseteq{H}$, and $(M/L)\ideal{h}=(M\ideal{h}+L)/L\subseteq{H/L}$. Otherwise, if $M/L$ is totally $\sigma$--simple with companion ideal $\ideal{h}\in\mathcal{L}(\sigma)$, and $H\subseteq{M}$ is not totally $\sigma$--torsion, then $(H+L)/L$ is not totally $\sigma$--torsion, hence $(M/L)\ideal{h}\subseteq(H+L)/L$, and $M\ideal{h}+L\subseteq{H+L}$. Since $L$ is totally $\sigma$--torsion, there exists $\ideal{h}'\in\mathcal{L}(\sigma)$ such that $L\ideal{h}'=0$. Therefore, $M\ideal{h}\ideal{h}'\subseteq{H\ideal{h}'}\subseteq{H}$, and $M$ is totally $\sigma$--simple.
\par
If $M\ideal{h}$ is the core of $M$, its image is $M\ideal{h}/(M\ideal{h}\cap{L})$, which is totally $\sigma$--simple and every proper submodule is totally $\sigma$--torsion. Indeed, if $H/(M\ideal{h}\cap{L})\subsetneqq{M\ideal{h}/(M\ideal{h}\cap{L})}$, then $H\subsetneqq{M}$, so it is totally $\sigma$--torsion, hence $H/(M\ideal{h}\cap{L})$ is.
\end{proof}

An $A$--module $M$ is \textbf{core totally $\sigma$--simple} whenever $M$ is the core of a totally $\sigma$--simple module, i.e., whenever $M$ is not totally $\sigma$--torsion and every proper submodule is totally $\sigma$--torsion.

Observe that if $M$ is core totally $\sigma$--simple, then $M\ideal{h}=M$ for every $\ideal{h}\in\mathcal{L}(\sigma)$. Indeed, if $\ideal{h}_0\in\mathcal{L}(\sigma)$ is the companion ideal of $M$, then $M\ideal{h}_0=M$, and for every $\ideal{h}\in\mathcal{L}(\sigma)$ we have: $M\ideal{h}=M\ideal{h}\ideal{h}_0=M$.

If $M$ is a core totally $\sigma$--simple $A$--module, we have two different cases:
\begin{enumerate}[(1)]\label{pg:20191204}\sepa
\item
$\sigma{M}\neq{M}$. In this case, $\sigma{M}$ is totally $\sigma$--torsion and $M/\sigma{M}$ is $\sigma$--torsionfree and core totally $\sigma$--simple, hence it is simple. Indeed, every proper submodule is totally $\sigma$--torsion, hence zero.
\item
$\sigma{M}=M$. In this case, since every proper quotient is $\sigma$--torsion, it follows that $M$ has no simple quotients. In particular, $M$ is not finitely generated.
\end{enumerate}

Let $M$ be a core totally $\sigma$--simple $A$--module, we have the following two possibilities:
\begin{enumerate}[(1)]\sepa
\item
$M$ is not cyclic. Since every proper submodule of $M$ is totally $\sigma$--torsion, we have $M$ is $\sigma$--torsion. Hence $M$ is not finitely generated because it is not totally $\sigma$--torsion.
\item
$M$ is cyclic. Since it is not totally $\sigma$--torsion, then $\sigma{M}\neq{M}$, and $M$ has simple $\sigma$--torsionfree quotients. Otherwise, every simple quotient of $M$ is $\sigma$--torsionfree. Consequences of this fact are: $\sigma(M)$ is totally $\sigma$--torsion, and $\sigma{M}$ is contained in the intersection of all maximal submodules of $M$.
\end{enumerate}

\begin{lemma}
Let $A$ be a ring and $\ideal{a}\subseteq{A}$ be a proper ideal, then $M=A/\ideal{a}$ is core totally $\sigma$--simple if, and only if, it satisfies:
\begin{enumerate}[(1)]\sepa
\item
$\ideal{a}\notin\mathcal{L}(\sigma)$,
\item
$\ideal{a}+\ideal{h}=A$, for every $\ideal{h}\in\mathcal{L}(\sigma)$, and
\item
$(\ideal{a}:\ideal{b})\in\mathcal{L}(\sigma)$, for every proper ideal $\ideal{b}\supseteq\ideal{a}$.
\end{enumerate}
\end{lemma}
\begin{proof}
If $A/\ideal{a}$ is core totally $\sigma$--simple, then it is not totally $\sigma$--torsion, i.e., $\ideal{a}\notin\mathcal{L}(\sigma)$. For any $\ideal{h}\in\mathcal{L}(\sigma)$ we have $(A/\ideal{a})\ideal{h}=A/\ideal{a}$, i.e., $\ideal{a}+\ideal{h}=A$. Since every proper submodule of $A/\ideal{a}$ is totally $\sigma$--torsion, then for every $\ideal{b}\supseteq\ideal{a}$ we have $(\ideal{a}:\ideal{b})\in\mathcal{L}(\sigma)$.
\end{proof}

In particular, for any maximal ideal $\ideal{m}$ the simple $A$--module $A/\ideal{m}$ is totally $\sigma$--simple if, and only if, $\ideal{m}\notin\mathcal{L}(\sigma)$.

\begin{lemma}
The class $\mathcal{M}$ is $\sigma$--lower closed.
\end{lemma}
\begin{proof}
Indeed, if $H\subseteq{M}$ and there are $N\in\mathcal{M}$ and $\ideal{h}\in\mathcal{L}(\sigma)$ such that $N\ideal{h}\subseteq{H}$, then $H$ is not totally $\sigma$--torsion. On the contrary there exists $\ideal{h}'\in\mathcal{L}(\sigma)$ such that $H\ideal{h}'=0$, hence $H\ideal{h}\ideal{h}'=0$, and $N$ is totally $\sigma$--torsion.
\end{proof}

If $M$ is totally $\sigma$--artinian, there are minimal elements in $\mathcal{M}$; every minimal element $N$ of $\mathcal{M}$ is a core totally $\sigma$--simple module, i.e., it satisfies:
\begin{enumerate}[(1)]\sepa
\item
$N$ is not totally $\sigma$--torsion, and
\item
Every proper submodule of $N$ is totally $\sigma$--torsion.
\end{enumerate}

\begin{lemma}
Let $M$ be a totally $\sigma$--simple module, for any submodule $N\subseteq{M}$ we have:
\begin{enumerate}[(1)]\sepa
\item
If $N$ is not totally $\sigma$--torsion, then $N$ is totally $\sigma$--simple. In addition, $N$ and $M$ have the same core.
\item
The quotient $M/N$ is either
totally $\sigma$--torsion whenever $N$ is non totally $\sigma$--torsion, or
totally $\sigma$--simple whenever $N$ is totally $\sigma$--torsion.
\end{enumerate}
\end{lemma}
\begin{proof}
  (1) is straightforward.
  \par
  (2). %
  If $N$ is non totally $\sigma$--torsion, there exists $\ideal{h}\in\mathcal{L}(\sigma)$ such that $M\ideal{h}\subseteq{N}$, hence $M/N$ is totally $\sigma$--torsion. Otherwise, if $N$ is totally $\sigma$--torsion, then $M/N$ is non totally $\sigma$--torsion; on the other hand, for any non totally $\sigma$--torsion submodule $L/N\subseteq{M/N}$, since $L\subseteq{M}$ is non totally $\sigma$--torsion, there exists $\ideal{h}\in\mathcal{L}(\sigma)$ such that $M\ideal{h}\subseteq{L}$; hence $(M/N)\ideal{h}\subseteq(L/N)$, and $M/N$ is totally $\sigma$--simple.
\end{proof}

\begin{lemma}
If $N_1$, $N_2$ are totally $\sigma$--simple modules, and $f:N_1\longrightarrow{N_2}$ is a module map, then $f$ is either zero or surjective.
\end{lemma}

\section{Maximal submodules}\label{se:20200102}

In this section we assume the reader knows about totally $\sigma$--noetherian rings and modules as it was exposed in \cite{Jara:2020a}.

Let $M$ be a totally $\sigma$--noetherian $A$--module, and consider the family of submodules
\[
\mathcal{M}'=\{N\subseteq{M}\mid\;M/N\textrm{ is not totally $\sigma$--torsion}\}.
\]
We have that $\mathcal{M}'$ is nonempty whenever $M$ is not totally $\sigma$--torsion.

\begin{lemma}
Let $M$ be a non totally $\sigma$--torsion module, the class $\mathcal{M}'$ is $\sigma$--upper closed.
\end{lemma}
\begin{proof}
If $H\subseteq{M}$ and there are $N\in\mathcal{M}'$ and $\ideal{h}\in\mathcal{L}(\sigma)$ such that $H\ideal{h}\subseteq{N}$, then $H\in\mathcal{M}'$. We show that $M/H$ is not totally $\sigma$--torsion. On the contrary, there exists $\ideal{h}_1\in\mathcal{L}(\sigma)$ such that $M\ideal{h}_1\subseteq{H}$; hence $M\ideal{h}_1\ideal{h}\subseteq{H\ideal{h}}\subseteq{N}$, and $M/N$ is totally $\sigma$--torsion, which is a contradiction.
\end{proof}

Since $M$ is totally $\sigma$--noetherian, there are maximal elements in $\mathcal{M}$; a maximal element of $\mathcal{M}'$ is called a totally $\sigma$--maximal submodule of $M$. We define a submodule $N$ of $M$ to be a \textbf{core totally $\sigma$--maximal submodule} whenever it satisfies:
\begin{enumerate}[(1)]\sepa
\item
$M/N$ is not totally $\sigma$--torsion.
\item
$N$ is maximal in the $\sigma$--upper closed family
$$
\mathcal{N}'=\{L\subseteq{M}\mid\;N\subseteq{L},\,M/L\textrm{ is not totally $\sigma$--torsion}\}.
$$
\end{enumerate}

In the same way, we can define a \textbf{totally $\sigma$--maximal submodule} of an $A$--module $M$ whenever
\begin{enumerate}[(1)]\sepa
\item
$M/N$ is not totally $\sigma$--torsion, and
\item
$N$ is $\sigma$--maximal in $\mathcal{N}'$, i.e., there exists $\ideal{h}\in\mathcal{L}(\sigma)$ such that for every $H\in\mathcal{N}'$ such that $N\subseteq{H}$ we have $H\ideal{h}\subseteq{N}$.
\end{enumerate}

Even, we can dualize the notion of totally $\sigma$--simple submodule, in defining a submodule $N\subseteq{M}$ to be \textbf{totally $\sigma$--cosimple} if it satisfies:
\begin{enumerate}[(1)]\sepa
\item
$M/N$ is not totally $\sigma$--torsion, and
\item
There exists $h\in\mathcal{L}(\sigma)$ such that for every $N\subseteq{H}\subsetneqq{M}$ satisfying that $H/N$ is not totally $\sigma$--torsion we have $M\ideal{h}\subseteq{H}$.
\end{enumerate}
and defining \textbf{core totally $\sigma$--cosimple} if, in addition, for every $N\subseteq{H}\subsetneqq{M}$ we have that $H/N$ is totally $\sigma$--torsion.

\begin{lemma}\label{le:20191202}
Let $M$ be a core totally $\sigma$--simple $A$--module, then $\Ann(M)\subseteq{A}$ is a core totally $\sigma$--cosimple ideal.
\end{lemma}
\begin{proof}
Since $M$ is core totally $\sigma$--simple we have two possibilities for $M$:
\begin{enumerate}[(1)]\sepa
\item\label{it:le:20191202-1}
$\sigma{M}=M$, and
\item\label{it:le:20191202-2}
$\sigma{M}\neq{M}$.
\end{enumerate}
\par
In case \eqref{it:le:20191202-2} there exists $x\in{M}\setminus\sigma{M}$, such that $xA=M$. Indeed, since $xA\subseteq{M}$ and it is not totally $\sigma$--torsion, then $xa=M$. Hence $\Ann(M)=\Ann(x)$, and $A/\Ann(x)$ is not totally $\sigma$--torsion. The rest is obvious.
\par
In case \eqref{it:le:20191202-1} we have that $M$ is cyclic and we can proceed in the same way.
\end{proof}

\begin{proposition}
Let $M$ be a totally $\sigma$--simple $A$--module, then $\Ann(M)\subseteq{A}$ is totally $\sigma$--cosimple.
\end{proposition}
\begin{proof}
Let $M\ideal{h}\subseteq{M}$ the core of $M$, we have $\Ann(M\ideal{h})=\Ann(M):\ideal{h}$, hence we can build a short exact sequence
$0
\to\dfrac{\Ann(M):\ideal{h}}{\Ann(M)}
\to\dfrac{A}{\Ann(M)}
\to\dfrac{A}{\Ann(M\ideal{h})}
\to0$. Since $\dfrac{\Ann(M):\ideal{h}}{\Ann(M)}$ is totally $\sigma$--torsion, we have the result.
\end{proof}

If $M=A$ we may determine more precisely the core totally $\sigma$--cosimple ideals.

\begin{proposition}\label{pr:20191205a}
Let $\ideal{a}\subseteq{A}$ be an ideal. If $\ideal{a}\subseteq{A}$ is core totally $\sigma$--cosimple then:
\begin{enumerate}[(1)]\sepa
\item
$\ideal{a}$ is prime.
\item
$\ideal{a}$ is maximal in $\mathcal{K}(\sigma)$, i.e., $\ideal{a}\in\mathcal{C}(\sigma)$.
\end{enumerate}
In conclusion, core totally $\sigma$--cosimple ideals are exactly the ideals in $\mathcal{C}(\sigma)$.
\end{proposition}
\begin{proof}
Since $\ideal{a}\subseteq{A}$ is core totally $\sigma$--cosimple, it is not totally $\sigma$--torsion, hence $\ideal{a}\notin\mathcal{L}(\sigma)$.
\par
(1). %
Let $\ideal{a}_1,\ideal{a}_2\subseteq{A}$ be proper ideals properly containing  $\ideal{a}$ such that $\ideal{a}_1\ideal{a}_2\subseteq\ideal{a}$. Since $\ideal{a}\subsetneqq\ideal{a}_1\subsetneqq{A}$, then $A/\ideal{a}_1$ is totally $\sigma$--torsion, hence $\ideal{a}_1\in\mathcal{L}(\sigma)$. Similar result holds for $\ideal{a}_2$. Therefore, $\ideal{a}_1\ideal{a}_2\in\mathcal{L}(\sigma)$, and $\ideal{a}\in\mathcal{L}(\sigma)$, which is a contradiction.
\par
(2). %
Since $\ideal{a}\notin\mathcal{L}(\sigma)$ is prime, then $\ideal{a}\in\mathcal{K}(\sigma)$. Let $\ideal{a}\subseteq\ideal{p}\in\mathcal{K}(\sigma)$, since $A/\ideal{p}$ is $\sigma$--torsionfree and non--zero, it is not totally $\sigma$--torsion, hence $\ideal{p}=\ideal{a}$. Therefore, $\ideal{a}\in\mathcal{K}(\sigma)$ is maximal and $\ideal{a}\in\mathcal{C}(\sigma)$.
\par
The converse is obvious because for any $\ideal{p}\in\mathcal{C}(\sigma)$ we have that $A/\ideal{p}$ is $\sigma$--cocritical.
\end{proof}

The \textbf{core totally $\sigma$--Jacobson radical} of an $A$--module $M$ is defined as the intersection of all core totally $\sigma$--cosimple submodule, and we represent it by $t\Jac(M)$.

\begin{lemma}
Let $M$ be an $A$--module, then $t\Jac\left(M/t\Jac(M)\right)=0$.
\end{lemma}
\begin{proof}
We have $N\subseteq{M/t\Jac(M)}$ is core totally $\sigma$--cosimple if, and only if, $N\subseteq{M}$ is core totally $\sigma$--cosimple.
\end{proof}

\medskip
\begin{proposition}\label{pr:20191214c}
\begin{enumerate}[(1)]\sepa
\item\label{it:pr:20191214c-1}
Every totally $\sigma$--simple $A$--module is totally $\sigma$--artinian.
\item\label{it:pr:20191214c-2}
Every core totally $\sigma$--simple $A$--module is totally $\sigma$--noetherian.
\item\label{it:pr:20191214c-3}
Every totally $\sigma$--simple $A$--module is totally $\sigma$--noetherian.
\end{enumerate}
\end{proposition}
\begin{proof}
\eqref{it:pr:20191214c-1}. %
Let $M$ be totally $\sigma$--simple and $N_1\supseteq{N_2}\supseteq\cdots$ be a decreasing chain of submodules of $M$. If for every index $i$ we have that $N_i$ is not totally $\sigma$--torsion, and $\ideal{h}\in\mathcal{L}(\sigma)$ is the companion ideal of $M$, then $M\ideal{h}\subseteq{N_i}$, hence $N_1\ideal{h}\subseteq{M\ideal{h}}\subseteq{N_i}$, for every index $i$. If there exists an index $m$ such that $N_m$ is totally $\sigma$--torsion, there exists $\ideal{h}'\in\mathcal{L}(\sigma)$ such that $N_m\ideal{h}'=0\subseteq{N_s}$, for every $s\geq{m}$.
\par
\eqref{it:pr:20191214c-2}. %
If $N_1\subseteq{N_2}\subseteq\cdots$ is an ascending chain of submodules of $M$, and $M$ is core totally $\sigma$--simple we studied the two cases in page~\pageref{pg:20191204}.
\par
(2.1). %
If $M$ is $\sigma$--torsion, then $M$ is cyclic, hence $\cup_nN_n\subseteq{M}$ is either totally $\sigma$--torsion or $\cup_nN_n=N$. In the second case, there exists an index $m$ such that $M=N_m$. In both cases the chain is $\sigma$--stable.
\par
(2.2). %
If $M$ is not $\sigma$--torsion, then $\sigma{M}$ is totally $\sigma$--torsion. If $\cup_nN_n\subseteq\sigma{M}$, then the chain $\sigma$-stabilizes. If there exists an index $m$ such that $N_m\nsubseteq\sigma{M}$, and $\ideal{h}\in\mathcal{L}(\sigma)$ is the companion ideal of $M$, then $M\ideal{h}\subseteq{N_m}$, and in this case the chain stabilizes.
\par
\eqref{it:pr:20191214c-3}. %
If $M$ is totally $\sigma$--simple, with companion ideal $\ideal{h}\in\mathcal{L}(\sigma)$, then $M\ideal{h}$ is core totally $\sigma$--simple, hence it is totally $\sigma$--noetherian. Otherwise, $M/M\ideal{h}$ is totally $\sigma$--torsion, hence totally $\sigma$--noetherian. Therefore, $M$ is totally $\sigma$--noetherian because it is an extension of $M\ideal{h}$ by $M/M\ideal{h}$.
\end{proof}

\begin{lemma}\label{le:20191214d}
Let $M$ be an artinian $A$--module such that $t\Jac(M)=0$, there exists $T\subseteq{M}$, totally $\sigma$--torsion such that $M/T\subseteq\oplus_{j=1}^tS_j$, for a finite family of core totally $\sigma$--simple $A$--modules. In particular, $M$ is totally $\sigma$--noetherian.
\end{lemma}
\begin{proof}
If $t\Jac(M)=0$ the intersection of all core totally $\sigma$--cosimple submodules is zero, and since $M$ is $\sigma$--finitely cogenerated, there exists a finite family of core totally $\sigma$--cosimple submodules, $\{N_1,\ldots,N_t\}$ such that $\cap_{j=1}^tN_j$ is totally $\sigma$--torsion. If we call $T=\cap_{j=1}^tN_i$, then $M/T$ is a submodule of $\oplus_{j=1}^t(M/N_j)$. Finally, since $\oplus_{j=1}^t(M/N_j)$ is totally $\sigma$--noetherian, then $M/T$ is totally $\sigma$--noetherian, and we have $M$ is totally $\sigma$--noetherian.
\end{proof}

\begin{theorem}
If $A$ is a totally $\sigma$--artinian ring, then $A$ is totally $\sigma$--noetherian.
\end{theorem}
\begin{proof}
Let $J=t\Jac(A)$. If $A$ is totally $\sigma$--artinian then $A/J$ is totally $\sigma$--artinian and totally $\sigma$--noetherian, by Lemma~\eqref{le:20191214d}.
We consider $A/J^2$, because for every ideal $\ideal{a}\subseteq{A}$ we have that $\ideal{a}\subseteq{A}$ is a core totally $\sigma$--cosimple ideal if, and only if, $\ideal{a}/J^2\subseteq{A/J^2}$ is core totally $\sigma$--cosimple, then $t\Jac(A/J^2)=J/J^2$, and the same holds for every $m\in\mathbb{N}$, i.e, $t\Jac(A/J^m)=J/J^m$.
\par
The decreasing chain $J=J^1\supseteq{J^2}\supseteq\cdots$ is $\sigma$--stable, hence there exist $m\in\mathbb{N}$ and $\ideal{h}\in\mathcal{L}(\sigma)$ such that $J^m\ideal{h}\subseteq{J^s}$ for every $s\geq{m}$.
We do induction on $m$. Let us assume $m=1$, then $J/J^2$ is totally $\sigma$--torsion, and we have a short exact sequence
\[
0\longrightarrow
J/J^2\longrightarrow
A/J^2\longrightarrow
A/J\longrightarrow0
\]
Since $J^1/J^2$ and $A/J$ are totally $\sigma$--artinian and totally $\sigma$--noetherian, then $A/J^2$ is. We assume the result holds for any positive integral number smallest than $m$ and that $J^m\ideal{h}\subseteq{J^s}$ for every $s\geq{m}$.
Consider the short exact sequence
\[
0\longrightarrow
J^m/J^{m+1}\longrightarrow
A/J^{m+1}\longrightarrow
A/J^m\longrightarrow0
\]
Since $J^m/J^{m+1}$ is totally $\sigma$--torsion and $A/J^m$ is totally $\sigma$--artinian and totally $\sigma$--noethe\-rian then $A/J^{m+1}$ is.
\end{proof}

For any ring $A$ the totally Jacobson $\sigma$--radical of $A$ which is the intersection of all core totally $\sigma$--cosimple ideals is the intersection of all elements in $\mathcal{C}(\sigma)$, see Proposition~\eqref{pr:20191205a}, which coincides with the \textbf{Jacobson $\sigma$--radical} of $A$, i.e.,
\[
t\Jac(A)=\Jac_\sigma(A)=\cap\{\ideal{p}\mid\;\ideal{p}\in\mathcal{C}(\sigma)\}.
\]

We show that there exist enough core totally $\sigma$--cosimple submodule in the following sense.

\begin{lemma}
Let $\sigma$ be a finite type hereditary torsion theory, for any totally $\sigma$--finitely generated module $M$ and any proper submodule $N\subseteq{M}$ such that $M/N$ is not totally $\sigma$--torsion, there exists a core totally $\sigma$--cosimple submodule $N\subseteq{H}\subsetneqq{M}$.
\end{lemma}
\begin{proof}
  Let $\Gamma=\{H\subseteq{M}\mid\;N\subseteq{H}\subseteq{M}\textrm{ and $M/H$ is not totally $\sigma$--torsion}\}$. If $\Gamma=\{N\}$, then $N$ is maximal among those submodules which are not totally $\sigma$--torsion, hence it is core totally $\sigma$--cosimple. Otherwise, for any increasing chain $N_1\subseteq{N}_2\subseteq\cdots$ in $\Gamma$, we consider $\cup_{n\geq1}N_n$. If $M/\cup_{n\geq1}N_n$ is totally $\sigma$--torsion, there exists $\ideal{h}\in\mathcal{L}(\sigma)$, finitely generated, such that $M\ideal{h}\subseteq\cup_{n\geq1}N_n$, and there is an index $m$ such that $M\ideal{h}\subseteq{N_m}$. Otherwise, there exist $N'\subseteq{M}$, finitely generated, and $\ideal{h}'\in\mathcal{L}(\sigma)$, finitely generated, such that $M\ideal{h}'\subseteq{N'}$. In consequence, $M\ideal{h}'\ideal{h}\subseteq{N\ideal{h}}\subseteq{M\ideal{h}}\subseteq\cup_{n\geq1}N_n$, and there exists an index $m$ such that $M\ideal{h}'\ideal{h}\subseteq{N_m}$, which is a contradiction. Thus, $\Gamma$ is an inductive set of submodules, and by Zorn's lemma we have that $\Gamma$ has maximal submodules. A maximal submodule $N\subseteq{M}$ in $\Gamma$ is a core totally $\sigma$--cosimple submodule.
\end{proof}

Of particular interest is the case in which $M=A$; in this case for every ideal $\ideal{a}\subseteq{A}$ such that $\ideal{a}\notin\mathcal{L}(\sigma)$, there exist $\ideal{p}\in\mathcal{C}(\sigma)$ such that $\ideal{a}\subseteq\ideal{p}$.

\providecommand{\bysame}{\leavevmode\hbox to3em{\hrulefill}\thinspace}
\providecommand{\MR}{\relax\ifhmode\unskip\space\fi MR }
\providecommand{\MRhref}[2]{%
  \href{http://www.ams.org/mathscinet-getitem?mr=#1}{#2}
}
\providecommand{\href}[2]{#2}

\end{document}